\documentclass[11pt, reqno]{article}
\usepackage{amsmath,amssymb,amsfonts,amsthm}
\usepackage{color}

\usepackage[colorlinks=true,linkcolor=blue,citecolor=blue,urlcolor=blue,pdfborder={0 0 0}]{hyperref}
\usepackage{cleveref}

\usepackage{caption}
\usepackage{thmtools}

\usepackage{mathrsfs}

\crefname{theorem}{Theorem}{Theorems}
\crefname{thm}{Theorem}{Theorems}
\crefname{lemma}{Lemma}{Lemmas}
\crefname{lem}{Lemma}{Lemmas}
\crefname{remark}{Remark}{Remarks}
\crefname{prop}{Proposition}{Propositions}
\crefname{defn}{Definition}{Definitions}
\crefname{corollary}{Corollary}{Corollaries}
\crefname{conjecture}{Conjecture}{Conjectures}
\crefname{question}{Question}{Questions}
\crefname{chapter}{Chapter}{Chapters}
\crefname{section}{Section}{Sections}
\crefname{figure}{Figure}{Figures}
\crefname{example}{Example}{Examples}

\theoremstyle{plain}
\newtheorem{thm}{Theorem}[section]

\newtheorem{lemma}[thm]{Lemma}
\newtheorem{theorem}[thm]{Theorem}

\newtheorem{corollary}[thm]{Corollary}

\newtheorem{prop}[thm]{Proposition}

\theoremstyle{definition}

\theoremstyle{remark}
\newtheorem{remark}[thm]{Remark}

\numberwithin{equation}{section}

\renewcommand{\P}{\mathbb P}
\newcommand{\E}{\mathbb E}

\newcommand{\R}{\mathbb R}
\newcommand{\Z}{\mathbb Z}
\newcommand{\N}{\mathbb N}

\newcommand{\cF}{\mathcal F}

\newcommand{\sA}{\mathscr A}

\newcommand{\sC}{\mathscr C}

\newcommand{\bbH}{\mathbb H}

\usepackage[margin=1.02in]{geometry}
\usepackage{tikz}
\usepackage{pgfplots}
\usepackage{extarrows}
\usepackage{titling}
\usepackage{commath}
\usepackage{wasysym}
\usepackage{mathtools}
\usepackage{titlesec}
\newcommand{\eps}{\varepsilon}
\usepackage{bbm}
\usepackage{setspace}
\usepackage{enumitem}
\usepackage{booktabs}

\usepackage[textsize=tiny]{todonotes}

\usepackage{cite}

\def\P{\mathbb{P}}

\DeclareMathSymbol{\leqslant}{\mathalpha}{AMSa}{"36} 
\DeclareMathSymbol{\geqslant}{\mathalpha}{AMSa}{"3E} 
\DeclareMathSymbol{\eset}{\mathalpha}{AMSb}{"3F}     

\renewcommand{\epsilon}{\varepsilon}

\makeatletter
\pgfdeclareshape{slash underlined}
{
  \inheritsavedanchors[from=rectangle] 
  \inheritanchorborder[from=rectangle]
  \inheritanchor[from=rectangle]{north}
  \inheritanchor[from=rectangle]{north west}
  \inheritanchor[from=rectangle]{north east}
  \inheritanchor[from=rectangle]{center}
  \inheritanchor[from=rectangle]{west}
  \inheritanchor[from=rectangle]{east}
  \inheritanchor[from=rectangle]{mid}
  \inheritanchor[from=rectangle]{mid west}
  \inheritanchor[from=rectangle]{mid east}
  \inheritanchor[from=rectangle]{base}
  \inheritanchor[from=rectangle]{base west}
  \inheritanchor[from=rectangle]{base east}
  \inheritanchor[from=rectangle]{south}
  \inheritanchor[from=rectangle]{south west}
  \inheritanchor[from=rectangle]{south east}
  \inheritanchorborder[from=rectangle]
  \foregroundpath{
    \southwest \pgf@xa=\pgf@x \pgf@ya=\pgf@y
    \northeast \pgf@xb=\pgf@x \pgf@yb=\pgf@y
    \pgf@xc=\pgf@xa
    \advance\pgf@xc by .5\pgf@xb
    \pgf@yc=\pgf@ya
    \advance\pgf@xc by -1.3pt
    \advance\pgf@yc by -1.8pt
    \pgfpathmoveto{\pgfqpoint{\pgf@xc}{\pgf@yc}}
    \advance\pgf@xc by  2.6pt
    \advance\pgf@yc by  3.6pt
    \pgfpathlineto{\pgfqpoint{\pgf@xc}{\pgf@yc}}
    \pgfpathmoveto{\pgfqpoint{\pgf@xa}{\pgf@ya}}
    \pgfpathlineto{\pgfqpoint{\pgf@xb}{\pgf@ya}}
 }
}
\tikzset{nomorepostaction/.code=\let\tikz@postactions\pgfutil@empty}

\newcommand\nxleftrightarrow[2][]{%
  \mathrel{\tikz[baseline=-.7ex] \path node[slash underlined,draw,<->,anchor=south] {\(\scriptstyle #2\)} node[anchor=north] {\(\scriptstyle #1\)};}}
\makeatother

\pretitle{\begin{flushleft}\Large}
\posttitle{\par\end{flushleft}\vskip 0.5em
}
\preauthor{\begin{flushleft}}
\postauthor{
\\
\vspace{0.5em}
\footnotesize{Statslab, DPMMS, University of Cambridge\\ Email: 
\href{mailto:t.hutchcroft@maths.cam.ac.uk}{t.hutchcroft@maths.cam.ac.uk}}
\end{flushleft}}
\predate{\begin{flushleft}}
\postdate{\par\end{flushleft}}
\title{\bf The critical two-point function for long-range percolation on the hierarchical lattice}

\renewenvironment{abstract}
 {\par\noindent\textbf{\abstractname.}\ \ignorespaces}
 {\par\medskip}

\titleformat{\section}
  {\normalfont\large \bf}{\thesection}{0.4em}{}

\author{{\bf Tom Hutchcroft}}

\begin{document}

\date{\small{\today}}

\maketitle

\setstretch{1.1}

\begin{abstract} We prove up-to-constants bounds on the two-point function (i.e., point-to-point connection probabilities) for critical long-range percolation on the $d$-dimensional hierarchical lattice. More precisely, we prove that if we connect each pair of points $x$ and $y$ by an edge with probability  $1-\exp(-\beta\|x-y\|^{-d-\alpha})$, where $0<\alpha<d$ is fixed and $\beta\geq 0$ is a parameter, then the critical two-point function satisfies
\[
\P_{\beta_c}(x\leftrightarrow y) \asymp \|x-y\|^{-d+\alpha}
\]
for every pair of distinct points $x$ and $y$. We deduce in particular that the model has mean-field critical behaviour when $\alpha<d/3$ and does \emph{not} have mean-field critical behaviour when $\alpha>d/3$.



\end{abstract}

\section{Introduction}
\label{sec:intro}

In this paper we study critical long-range percolation on the \emph{hierarchical lattice}.  The hierarchical lattice $\bbH^d_L$ is in some ways similar to the usual Euclidean lattice $\Z^d$ but has additional symmetries and an exact recursive nesting structure that often makes hierarchical models of statistical mechanics much easier to understand than their Euclidean counterparts. 
First introduced by Dyson \cite{MR436850} to study the Ising model in 1969, there is now an extensive literature studying statistical mechanics on hierarchical lattices, with notable works studying the critical behaviour of the Ising model \cite{MR1552611,MR1882398}, the $\varphi^4$ model \cite{koch1994nontrivial,MR3969983}, and self-avoiding walk \cite{MR2000929,MR1143413,MR2000928}.
Hierarchical models have been particularly popular when e.g.\ analyzing spin systems via the renormalization group, where the exact recursive structure is extremely helpful \cite{MR709462,MR3969983}. We refer the reader to e.g.\ \cite{MR2883859,MR3969983} for further background on hierarchical models, and to \cite[Section 4.2]{MR3969983} in particular for a detailed overview of the literature.

While several papers have been written about Bernoulli percolation on hierarchical lattices \cite{MR2955049,MR3035740,MR3769822,MR4132522}, we think it is fair to say that hierarchical models have received rather less attention within percolation theory than within other parts of mathematical physics, 
 and one goal of this paper is to attract more interest to hierarchical models within the percolation community more broadly. To this end, let us note by way of advertisement that not only do hierarchical models offer a much more tractable alternative to low-dimensional Euclidean models,
   they are also arguably \emph{more realistic} than Euclidean models as descriptions of certain real-world phenomena in epidemiology and the social sciences \cite{MR3464199,gandolfi2013percolation}. Let us also remark that existing analyses of other hierarchical models at criticality are mostly based on block-renormalization of spins and therefore do not apply to percolation, which is not rigorously known to have any spin-system representations.




\begin{figure}[t!]
\centering
\includegraphics[height=9.8cm]{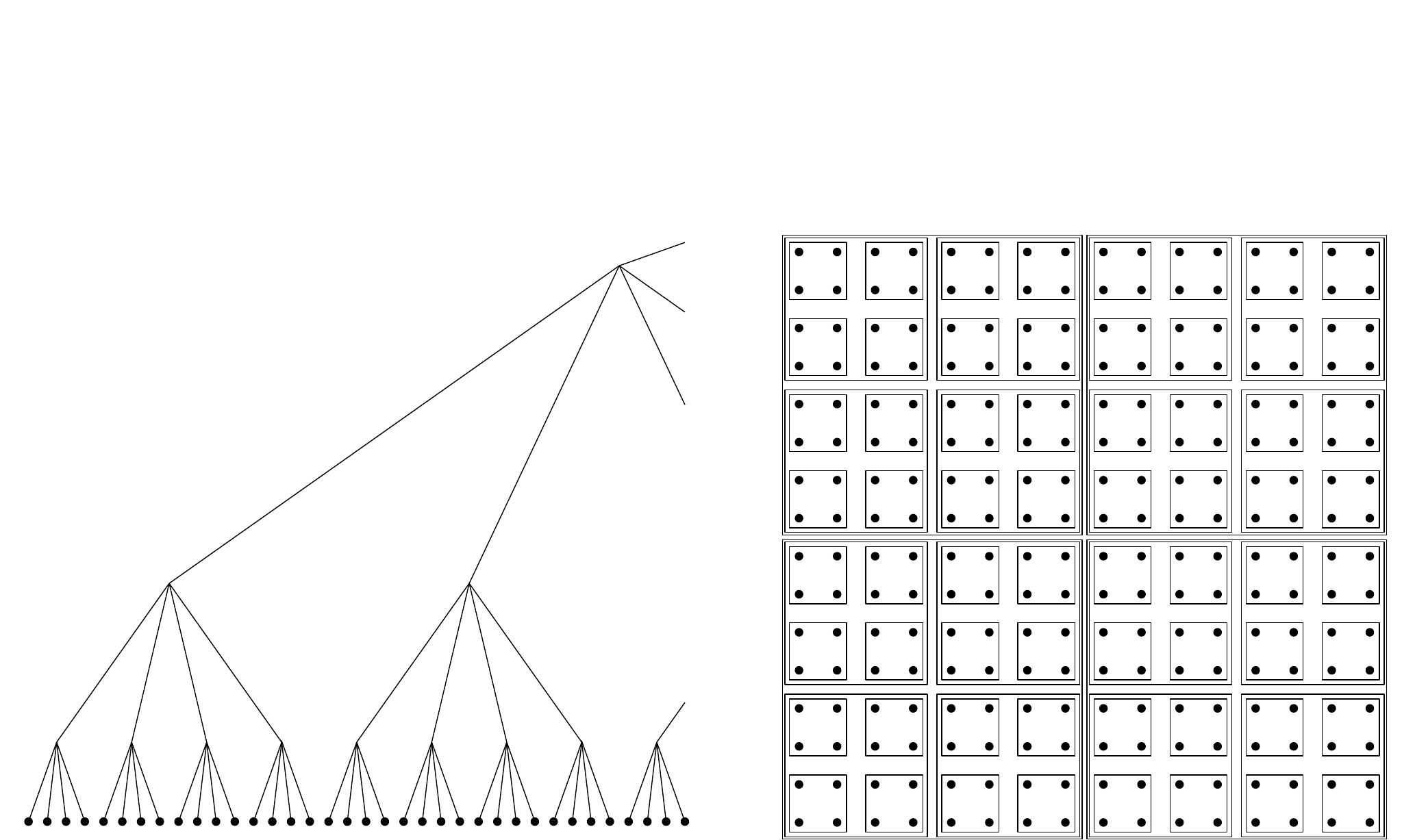}
\caption{Two graphical representations of the hierarchical lattice $\mathbb{H}^2_2$, which can be identified with $\mathbb{H}^1_4$ by a bijection that transforms distances by a power. In the picture on the left, only the leaves of the tree represent vertices of $\bbH^1_4$. In the picture on the right, the distance between two points is equal to the side-length of the smallest dyadic box containing both points.}
\end{figure}

\subsection{The model}

Let us now define the hierarchical lattice.
Let $d\geq 1$, $L\geq 2$, and let $\mathbb{T}^d_L=(\Z/L\Z)^d$ be the discrete $d$-dimensional torus of side length $L$. 
The \textbf{hierarchical lattice} $\mathbbm{H}^d_L$ is defined to be the countable abelian group $\bigoplus_{i=1}^\infty \mathbb{T}^d_L = \{x =(x_1,x_2,\ldots) \in (\mathbb{T}^d_L)^\N : x_i =0$ for all but finitely many $i\geq 0\}$ equipped with the group-invariant ultrametric defined by
\[\|y-x\| = 
\begin{cases} 0 & x=y\\
L^{h(x,y)} & x \neq y
\end{cases} \qquad \text{ where }h(x,y)=\max\{i \geq 1: x_i \neq y_i\}.
\]
We will also use the `Japanese bracket' notation $\langle x\rangle := 1\vee\|x\|$ to avoid diving by zero.
Note that the hierarchical lattice $\mathbbm{H}^d_L$ is indeed $d$-dimensional in the sense that if $B(0,r)$ denotes the ball of radius $r$ around the origin then
\[
\frac{r^d}{L^d} \leq  |B(0,r)| = |B(0,L^{\lfloor \log_L r\rfloor})| = \Bigl|\bigoplus_{i=1}^{\lfloor \log_L r\rfloor} \mathbb{T}^d_L\Bigr| = L^{d \lfloor \log_L r\rfloor}\leq r^d
\]
for every $r\geq 1$. As mentioned above, hierarchical lattices have much more symmetry than their Euclidean counterparts, and this symmetry can often be very useful when studying statistical-mechanics models on them. Indeed, the hierarchical lattice $\bbH^d_L$ is \textbf{distance-transitive}, meaning that if $w,x,y,z \in \bbH^d_L$ are such that $\|w-x\|=\|y-z\|$ then there exists an isometry $\gamma$ of $\bbH^d_L$ such that $\gamma(w)=y$ and $\gamma(x)=z$. 

A function $J: \bbH^d_L \to [0,\infty)$ is said to be \textbf{symmetric} if $J(x)=J(-x)$ for every $x\in \bbH^d_L$ and is said to be \textbf{integrable} if $\sum_{x\in \bbH^d_L} J(x)<\infty$. We say that $J$ is \textbf{radially symmetric} if $J(x)$ can be expressed as a function of $\|x\|$, i.e., if $J(x)=J(y)$ for every $x,y\in \bbH^d_L$ such that $\|x\|=\|y\|$. Equivalently, $J$ is radially symmetric if it is invariant under isometries of $\bbH^d_L$.
For our purposes, a particularly interesting choice of integrable, radially symmetric $J$ is given by
$J(x)=\langle x\rangle^{-d-\alpha}$, which is integrable for $\alpha>0$. 
Given a symmetric, integrable function $J:\bbH^d_L\to [0,\infty)$ and $\beta \geq 0$, \textbf{long-range percolation} on $\bbH^d_L$ is defined to be the random graph with vertex set $\bbH^d_L$ in which each pair $\{x,y\}$ is included as an edge of the graph independently at random with inclusion probability $1-e^{-\beta J(x-y)}$. (The precise form of this function is not very important; the important thing is that it belongs to $(0,1)$ when $J(x-y)$ is positive and satisfies $1-e^{-\beta J(x-y)} \sim \beta J(x-y)$ when $J(x-y)$ is small.) 
We write $\P_\beta=\P_{\beta,J}$ and $\E_\beta = \E_{\beta,J}$ for probabilities and expectations taken with respect to the law of the resulting random graph.
The integrability assumption on $J$ ensures that this graph is locally finite (i.e., has finite vertex degrees) almost surely. The connected components of the resulting random graph are known as \textbf{clusters} and the \textbf{critical probability} $\beta_c=\beta_c(d,L,J)$ is defined by
\[
\beta_c = \inf\bigl\{\beta \geq 0: \text{ there exists an infinite cluster with positive probability}\bigr\}.
\]
 Elementary path-counting arguments yield that $\beta_c \geq 1/\sum_x J(x)>0$ when $J$ is integrable.

 For $J$ of the form $J(x)=\langle x\rangle^{-d-\alpha}$ with $\alpha>0$, it has been shown independently by Koval, Meester, and Trapman \cite{MR2955049} and Dawson and Gorostiza \cite{MR3035740} that $\beta_c<\infty$ 
 (i.e., that the phase transition is non-trivial)
  if and only if $\alpha<d$. Koval, Meester, and Trapman \cite{MR2955049} have also shown under the same assumptions that the phase transition is \emph{continuous}, meaning that there is no infinite cluster at $\beta_c$ almost surely. This is a hierarchical version of a theorem of Berger \cite{MR1896880}, which establishes a similar result for long-range percolation on $\Z^d$ with $\alpha<d$. (In contrast, long-range percolation on $\Z$ with $\alpha=1$ is known to undergo a \emph{discontinuous} phase transition by a theorem of Aizenman and Newman \cite{MR868738}; see also \cite{duminil2020long} for a new proof of this result.)  Both theorems are proven by showing that the set $\{\beta>0 :$ there is an infinite cluster at $\beta\}$ is open, and do not yield any quantitative control of the models at criticality. 

In our recent work \cite{hutchcroft2020power} we established new, quantitative versions of both continuity theorems, yielding power-law upper bounds on the distribution of the cluster of the origin in critical long-range percolation on both $\Z^d$ and $\bbH^d_L$. The power-law bounds proven in \cite{hutchcroft2020power} are not expected to be sharp, and it remained open to compute the exact critical exponents describing these models.

\subsection{Our results}

 The goal of this paper is to establish a sharp quantitative understanding of critical long-range percolation on the hierarchical lattice by proving up-to-constants estimates on the two-point function (i.e., on the point-to-point connection probabilities $\P_\beta(x\leftrightarrow y)$). Before stating our main result, let us briefly introduce some further relevant definitions. 
For each $n\geq 0$ we write $\Lambda_n=\bigoplus_{i=1}^n \mathbb{T}^d_L = \{y\in \mathbb{H}_L^d: \langle y\rangle \leq L^n\}$ and write $\Lambda_n(x)=\Lambda_n + x = \{y\in \mathbb{H}_L^d: \langle y-x\rangle \leq L^n\}$ for each $n\geq 0$ and $x \in \mathbb{H}^d_L$. Given $x,y \in \bbH^d_L$ and $A \subseteq \bbH^d_L$ we write $\{x \xleftrightarrow{A} y\}$ for the event that $x$ and $y$ are connected by a path in $A$.

\begin{theorem}
\label{thm:main}
Let $J:\mathbb{H}^d_L \to [0,\infty)$ be a symmetric, integrable function, let $0<\alpha<d$, and suppose that there exist positive constants $c$ and $C$ such that $c \|x\|^{-d-\alpha} \leq J(x) \leq C\|x\|^{-d-\alpha}$ for every $x\in \mathbb{H}^d_L \setminus \{0\}$. Then there exist positive constants $a$ and $A$ depending only on $d$, $L$, $\alpha$, $c$, and $C$ such that
\[
a L^{-(d-\alpha)n}  \leq \frac{1}{|\Lambda_n|} \sum_{x\in\Lambda_n \setminus \Lambda_{n-1}}\P_{\beta_c}(0 \xleftrightarrow{\Lambda_n} x) \leq \frac{1}{|\Lambda_n|} \sum_{x\in\Lambda_n}\P_{\beta_c}(0 \leftrightarrow{} x) \leq A L^{-(d-\alpha)n}
\]
for every $n \geq 1$.
\end{theorem}

Note that this theorem implies in particular that $\beta_c<\infty$ and that there are no infinite clusters at $\beta_c$, recovering the results of \cite{MR2955049,MR3035740}.
For radially symmetric $J$, the averaged two-point function estimate of \cref{thm:main} can immediately be upgraded to a pointwise estimate via symmetry, which we now state.
  Given two points $x,y\in \mathbb{H}^d_L$ we write 
\[\Lambda(x,y) = \left\{z \in \mathbb{H}^d_L : \|z-x\| \leq \|y-x\|\right\}=\left\{z \in \mathbb{H}^d_L : \|z-y\| \leq \|y-x\|\right\}\]
for the smallest ultrametric ball containing both $x$ and $y$.

\begin{corollary}
\label{cor:mainsymmetric}
Let $J:\mathbb{H}^d_L \to [0,\infty)$ be a radially symmetric, integrable function, let $0<\alpha<d$, and suppose that there exist positive constants $c$ and $C$ such that $c \|x\|^{-d-\alpha} \leq J(x) \leq C\|x\|^{-d-\alpha}$ for every $x\in \mathbb{H}^d_L \setminus \{0\}$. Then there exist positive constants $a$ and $A$ depending only on $d$, $L$, $\alpha$, $c$, and $C$ such that
\[
a \langle x-y\rangle^{-d+\alpha} \leq 
\P_{\beta_c}(x \xleftrightarrow{\Lambda(x,y)} y) \leq
 \P_{\beta_c}(x \leftrightarrow y) \leq A \langle x-y\rangle^{-d+\alpha}
\]
for every $x,y\in \mathbb{H}^d_L$.
\end{corollary}

\begin{proof}[Proof of \cref{cor:mainsymmetric} given \cref{thm:main}]
Since $\bbH^d_L$ is distance-transitive, $\P_\beta(0\xleftrightarrow{\Lambda_n} x)$ depends only on $n$ and $\|x\|$, so that if $\|x\|=L^n$ then we have by \cref{thm:main} that
\begin{align*}\P_{\beta_c}(0 \xleftrightarrow{\Lambda_n} x) &= \frac{1}{|\Lambda_n \setminus \Lambda_{n-1}|}\sum_{y\in \Lambda_n \setminus \Lambda_{n-1}} \P_{\beta_c}(0 \xleftrightarrow{\Lambda_n} y) \geq aL^{-(d-\alpha)n} \qquad \text{and}
\\
\P_{\beta_c}(0 \xleftrightarrow{} x) &= \frac{1}{|\Lambda_n \setminus \Lambda_{n-1}|}\sum_{y\in \Lambda_n \setminus \Lambda_{n-1}} \P_{\beta_c}(0 \xleftrightarrow{} y) \leq 
 \frac{L^d}{L^d-1} A L^{-(d-\alpha)n}\end{align*}
as claimed.
\end{proof}

\begin{remark}
For long-range percolation on $\Z^d$ with $d\geq 2$, it is predicted that the critical two-point function satisfies the same asymptotics given by \cref{cor:mainsymmetric} for $\alpha$ strictly smaller than the \emph{crossover value} $\alpha_c=\alpha_c(d)$, while for $\alpha>\alpha_c$ the critical two-point function should obey the same asymptotic decay as for \emph{nearest neighbour} percolation on $\Z^d$. See \cite[Section 1.3]{hutchcroft2020power} for details. Here we see that the behaviour on the hierarchical lattice is rather simpler, and indeed is more closely analogous to long-range percolation on the \emph{one-dimensional} lattice $\Z$.
This is related to the fact that $\bbH^d_L$ and $\bbH^1_{L^d}$ are related by a bijection that transforms distances by a $d$th root, so that long-range percolation on $\bbH^d_L$ with $J(x)=\langle x \rangle^{-d-\alpha}$ is equivalent to long-range percolation on $\bbH^1_{L^d}$ with $J(x)= \langle x \rangle^{-1-\alpha/d}$. As such, the dimension $d$ does not really make any difference to the model besides a change of parameterization. 
\end{remark}

\textbf{Further critical exponents.}
 It would be very interesting to compute further critical exponents of the model beyond those describing the two-point function. The next most accessible of these critical exponents is probably the exponent $\delta$, which is conjectured to describe the tail of the volume of a critical cluster via
\[
\delta = -\lim_{n\to\infty} \frac{\log n}{\log \P_{\beta_c}(|K|\geq n)} \quad \text{ so that} \quad \P_{\beta_c}(|K|\geq n) = n^{-1/\delta \pm o(1)} \quad \text{ as $n\to\infty$,}
\]
where $|K|$ denotes the cluster of the origin.
(It is a part of the conjecture that this limit is well-defined.) Using \cref{thm:main}, heuristic hyperscaling arguments (see e.g.\ \cite[Sections 1.3 and 2]{hutchcroft2020power}) predict that, under the hypotheses of \cref{thm:main},
\begin{equation}
\label{eq:delta_conjecture}
\delta = \begin{cases} 2 & \text{if } \; 0< \alpha \leq d/3\\
\frac{d+\alpha}{d-\alpha} & \text{if }\;d/3< \alpha < d.
\end{cases} 
\end{equation}
In particular, it is expected that the model should have \emph{mean-field critical behaviour} if and only if $\alpha<d/3$, with polylogarithmic corrections to this behaviour when $\alpha=d/3$.

While we have not been able to prove \eqref{eq:delta_conjecture}, \cref{thm:main} does yield some interesting partial progress on the problem.
 Indeed, when $\alpha<d/3$ we easily verify from \cref{thm:main} that the model satisfies the \emph{triangle condition} of Aizenman and Newman \cite{MR762034}, which is known to imply that many critical exponents exist and take their mean-field values \cite{MR762034,MR1127713,MR2551766,MR2748397}. (These proofs are usually written for Euclidean lattices but apply equally well in the hierarchical case.) Since the triangle condition and its consequences are well-known, we do not go into them in detail here but refer the reader to e.g.\ \cite[Chapter 10.3]{grimmett2010percolation} and \cite{heydenreich2015progress} for background.

\begin{corollary}
\label{cor:triangle}
Let $J:\mathbb{H}^d_L \to [0,\infty)$ be a radially symmetric, integrable function, let $0<\alpha<d$, and suppose that there exist constants $c$ and $C$ such that $c \|x\|^{-d-\alpha} \leq J(x) \leq C\|x\|^{-d-\alpha}$ for every $x\in \mathbb{H}^d_L \setminus \{0\}$. If $\alpha<d/3$ then the triangle condition
\[
\nabla_{\beta_c} := \sum_{x,y\in \bbH^d_L} \P_{\beta_c}(0 \leftrightarrow x)\P_{\beta_c}(x \leftrightarrow y)\P_{\beta_c}(y \leftrightarrow 0) < \infty
\]
holds and the model exhibits mean-field critical behaviour. In particular, we have that
\begin{align*}
\P_{\beta_c}(|K| \geq n) &\asymp n^{-1/2} && \text{ for every $n\geq 1$,}\\
\E_\beta |K| &\asymp (\beta_c-\beta)^{-1} && \text{ for every $0<\beta<\beta_c$, and }\\
\P_\beta(|K|=\infty) &\asymp \max\{\beta-\beta_c,1\} && \text{ for every $\beta>  \beta_c$},
\end{align*}
where $\asymp$ denotes an equality holding to within multiplication by positive constants.
\end{corollary}

We remark that, following the breakthrough work of Hara and Slade \cite{MR1043524}, the triangle condition has been verified for a number of high-dimensional Euclidean percolation models using a technique known as the \emph{lace expansion} \cite{MR1043524,MR2430773,MR3306002,MR1959796,MR782962,fitzner2015nearest}, which is surveyed in \cite{heydenreich2015progress} and \cite{MR2239599}. We expect that it should also be possible to prove a version of \cref{cor:triangle} using the lace expansion, but that this would be much more involved than the proof we have given and would also need slightly stronger hypotheses, for example that the constant $L$ is large.

\medskip

In the case $\alpha>d/3$, we are able to deduce a sharp \emph{lower bound} on the exponent $\delta$ from \cref{thm:main} together with the rigorous hyperscaling inequality of \cite[Theorem 2.1]{hutchcroft2020power}.

\begin{corollary}
\label{thm:delta_lower}
Let $J:\mathbb{H}^d_L \to [0,\infty)$ be a symmetric, integrable function, let $0<\alpha<d$, and suppose that there exist constants $c$ and $C$ such that $c \|x\|^{-d-\alpha} \leq J(x) \leq C\|x\|^{-d-\alpha}$ for every $x\in \mathbb{H}^d_L \setminus \{0\}$. If $\alpha>d/3$ and the exponent $\delta$ is well-defined then 
\[
\delta \geq \frac{d+\alpha}{d-\alpha}>2.
\]
Thus, the model does \emph{not} have mean-field critical behaviour when $\alpha>d/3$.
\end{corollary}

\begin{remark}
We are not aware of any conjectured values of the critical exponents $\beta$, $\gamma$, and $\Delta$ describing \emph{near critical} percolation on the hierarchical lattice when $\alpha>d/3$. (See e.g.\ \cite[Chapter 9]{grimmett2010percolation} for definitions of these exponents.) It would be very interesting even to have even heuristic calculations of these exponents in this regime.
\end{remark}

\section{Proof}

\subsection{The maximum cluster size}
\label{sec:maximum}

Fix $d\geq 1$, $L\geq 2$ and $0<\alpha<d$. 
For each $n\geq 1$ and $x\in \bbH^d_L$ let $K_n(x)$ be the cluster of $x$ in the ultrametric ball $\Lambda_n(x)$ (i.e., the set of $y\in \Lambda_n(x)$ connected to $x$ by an open path contained in $\Lambda_n(x)$) and write $K_n=K_n(0)$ to lighten notation. We will also write $K(x)$ for the cluster of $x$ in $\bbH^d_L$ and write $K=K(0)$. In this section, we study the distribution of the size of the \emph{largest} cluster in an ultrametric ball
\[
|K^\mathrm{max}_n|:= \max\{ |K_n(x)| : x \in \Lambda_n\}.
\]
(Note that this is a slight abuse of notation since the largest cluster might not be unique, in which case $K^\mathrm{max}_n$ is not well-defined as a set.) Following \cite[Section 2]{hutchcroft2020power}, we define the \textbf{typical value} of $|K^\mathrm{max}_n|$ to be
\[
M_n=M_n(\beta):= \min\Bigl\{m \geq 1 : \P_\beta\left(|K^\mathrm{max}_n| \geq m\right) \leq \frac{1}{e}\Bigr\} \]
for each $\beta \geq 0$. Note that $M_n(\beta)$ is an increasing, continuous function of $\beta$ for each $n\geq 0$. Note also that $M_n \geq 2$, so that
\begin{equation}
\label{eq:atleasttypical}
\P_\beta\left(|K^\mathrm{max}_n| \geq \frac{1}{2}M_n\right) \geq \P_\beta\left(|K^\mathrm{max}_n| \geq M_n-1\right) \geq \frac{1}{e}
\end{equation}
and hence by symmetry that
\begin{equation}
\label{eq:susceptibility>max}
\E_\beta |K_n| = \frac{1}{|\Lambda_n|}\sum_{x\in \Lambda_n}\E_\beta |K_n(x)| \geq  \frac{M_n^2}{4|\Lambda_n|}\P_\beta\left(|K^\mathrm{max}_n| \geq \frac{1}{2}M_n\right) \geq \frac{M_n^2}{4e L^{dn}} 
\end{equation}
for every $\beta\geq 0$ and $n\geq 0$.

\medskip

The following theorem is a special case of \cite[Theorem 2.2]{hutchcroft2020power} and shows in particular that the maximum cluster size $|K_n^{\mathrm{max}}|$ is of the same order as its typical value $M_n$ with high probability. 

\begin{theorem}
\label{thm:universaltightness}
Let $J:\mathbb{H}^d_L \to [0,\infty)$ be a symmetric, integrable function, let $\beta \geq 0$, and let $M_n=M_n(\beta)$ for each $n\geq 0$.
 The inequalities
\begin{equation}
\P_\beta\Bigl(|K_n^\mathrm{max}| \geq \lambda M_n\Bigr) \leq \exp\left(-\frac{1}{9}\lambda \right)
\label{eq:BigClusterUnrooted}
\qquad \text{and} \qquad \P_\beta\Bigl(|K_n^\mathrm{max}| < \eps M_n \Bigr) \leq 27 \eps 
\end{equation}
hold for every $n\geq 0$, $\lambda \geq 1$ and $0<\eps \leq 1$. Moreover, the inequality
\begin{equation}
\label{eq:BigClusterRooted}
\P_\beta\Bigl(|K_n| \geq \lambda M_n\Bigr) \leq  \P_\beta\Bigl(|K_n| \geq  M_n\Bigr) \exp\left(1-\frac{1}{9}\lambda \right)
\end{equation}
holds for every $n\geq 0$ and $\lambda \geq 1$.
\end{theorem}

The main goal of this section is to prove an upper bound on the growth of $M_n(\beta_c)$. \cref{thm:universaltightness}, and specifically the inequality \eqref{eq:BigClusterRooted}, will allow us to make use of these asymptotics to study the growth of $\E_{\beta_c} |K_n|$ in the next subsection.

\begin{prop}
\label{prop:maximum_upper}
Let $J:\mathbb{H}^d_L \to [0,\infty)$ be a symmetric, integrable function, let $0<\alpha<d$, and suppose that there exists a positive constant $c$ such that $J(x) \geq c\|x\|^{-d-\alpha}$ for every $x\in \mathbb{H}^d_L \setminus \{0\}$.
Then there exists a constant $A=A(d,L,\alpha,c)$ such that
\[
M_n(\beta) \leq \frac{A}{\beta} L^{(d+\alpha)n/2}
\]
for every $0\leq \beta \leq \beta_c$ and $n\geq 0$.
\end{prop}

Since $M_n(\beta) \geq 2$ for every $n\geq 0$ and $\beta \geq 0$, \cref{prop:maximum_upper} has the following immediate corollary. (Indeed, taking $n=0$ gives that $\beta_c \leq A/2$ where $A$ is the constant from \cref{prop:maximum_upper}.)

\begin{corollary}
\label{cor:betabound} 
Let $J:\mathbb{H}^d_L \to [0,\infty)$ be a symmetric, integrable function, let $0<\alpha<d$, and suppose that there exists a positive constant $c$ such that $J(x) \geq c\|x\|^{-d-\alpha}$ for every $x\in \mathbb{H}^d_L \setminus \{0\}$. Then there exists a constant $\beta_1=\beta_1(d,L,\alpha,c)$ such that $\beta_c \leq \beta_1$.
\end{corollary}

We will deduce \cref{prop:maximum_upper} from the sharpness of the phase transition together with the following renormalization lemma.

\begin{lemma}[Renormalization of the maximum cluster size]
\label{lem:maximum_renormalization}
Let $J:\mathbb{H}^d_L \to [0,\infty)$ be a symmetric, integrable function, let $0<\alpha<d$, and suppose that there exists a positive constant $c$ such that $J(x) \geq c\|x\|^{-d-\alpha}$ for every $x\in \mathbb{H}^d_L\setminus\{0\}$.
There exist constants $\ell=\ell(d,L,\alpha,c)$ and $A=A(d,L,\alpha,c)$ such that the implication
\[\left(M_n(\beta)^2 \geq \frac{A}{\beta} L^{(d+\alpha)n}\right) \Rightarrow \left(M_{n+\ell}(\beta)^2 \geq \frac{A}{\beta} L^{(d+\alpha)(n+\ell)} \right)\]
holds for every $\beta > 0$ and $n\geq 0$.
\end{lemma}



\begin{proof}[Proof of \cref{lem:maximum_renormalization}]
Since $\alpha<d$, there exists $\ell_0$ such that $\frac{1}{8}L^{d\ell} \geq L^{(d+\alpha)\ell/2}$ for every $\ell\geq \ell_0$.
The set $\Lambda_{n+\ell}$ contains $L^{d\ell}$ disjoint copies of $\Lambda_n$, each of which contains a cluster of size at least $M_n-1 \geq \frac{1}{2}M_n$ independently with probability at least $1/e>1/4$ by \eqref{eq:atleasttypical}. 
Letting $\sA_{n,\ell}$ be the event that at least $\frac{1}{4}L^{d\ell}$ of these copies of $\Lambda_n$ contains a cluster of size at least $\frac{1}{2}M_n$, it follows by standard concentration estimates for binomial random variables (or indeed by the weak law of large numbers) that we may take $\ell \geq \ell_0$ to be a sufficiently large constant that
$\P(\sA_{n,\ell}) \geq 9/10$ for every $n\geq 0$. We now fix $\ell=\ell(d,L,\alpha)$ to be one such constant.

Fix $n\geq 0$ and let $\cF_n$ be the $\sigma$-algebra generated by those edges of $\bbH^d_L$ that have endpoints at distance at most $L^n$. 
Condition on $\cF_n$ and suppose that $\sA_{n,\ell}$ holds, noting that $\sA_{n,\ell}$ is measurable with respect to $\cF_n$. Since $\sA_{n,\ell}$ holds there exist at least $\frac{1}{4} L^{d\ell}$ copies of $\Lambda_n$ in $\Lambda_{n+\ell}$ containing a cluster of size at least $\frac{1}{2} M_{n}$. Pick one such cluster within each of these copies, and call these clusters $K^1,\ldots,K^m$ where $\frac{1}{4} L^{d\ell} \leq m \leq L^{d\ell}$. We have by the definitions if $\sA_{n,\ell}$ holds then
\begin{align*}
\P_\beta(K^i \xleftrightarrow{\Lambda_{n+\ell}} K^j \mid \cF_n) &\geq \P_\beta(\text{$\exists$ an open edge connecting $K^i$ to $K^j$} \mid \cF_n)
\\
&=
1-\exp\left[ -\beta \sum_{x\in K^i}\sum_{y\in K^j} J(y-x) \right] \geq 1-\exp\left[-\frac{c\beta}{4} M_n^2 L^{-(d+\alpha)(n+\ell)} \right]
\end{align*}
for every $1\leq i  < j \leq m$. It follows by a union bound that
\begin{align*}
\P_\beta\left( |K^\mathrm{max}_{n+\ell}| \geq \frac{1}{2} M_n \cdot \frac{1}{4} L^{d\ell} \mid \cF_n\right) 
&\geq 
\P_\beta\left( |K^\mathrm{max}_{n+\ell}| \geq \sum_{i=1}^m |K^i|\mid \cF_n\right) \mathbbm{1}(\sA_{n,\ell})\\
&
\geq \left(1-\sum_{1\leq i < j \leq m} \P_\beta(K^i \nxleftrightarrow{\Lambda_{n+\ell}} K^j \mid \cF_n)\right)\mathbbm{1}(\sA_{n,\ell})\\
&\geq \left(1- L^{2d\ell} \exp\left[-\frac{c\beta}{4} M_n^2 L^{-(d+\alpha)(n+\ell)} \right]\right)\mathbbm{1}(\sA_{n,\ell})
\end{align*}
and hence that
\begin{equation}
\P_\beta\left( |K^\mathrm{max}_{n+\ell}| \geq \frac{1}{2} M_n \cdot \frac{1}{4} L^{d\ell} \right) \geq \frac{9}{10}\left(1- L^{2d\ell} \exp\left[-\frac{c\beta}{4} M_n^2 L^{-(d+\alpha)(n+\ell)} \right]\right).
\end{equation}
It follows that there exists a constant $A=A(d,L,\alpha,c)$ such that if  $M_n^2 \geq \frac{A}{\beta} L^{(d+\alpha)n}$ then
\[
\P_\beta\left( |K^\mathrm{max}_{n+\ell}| \geq \frac{1}{2} M_n \cdot \frac{1}{4} L^{d\ell} \right)  \geq \frac{9}{10} \left(1- L^{2d\ell} \exp\left[-\frac{c A}{4} L^{-(d+\alpha)\ell} \right]\right) > \frac{1}{e}
\]
  and hence that 
\begin{equation}M_{n+\ell} \geq \frac{1}{8} L^{d\ell}M_n \geq L^{(d+\alpha)\ell/2} M_n \geq \sqrt{\frac{A}{\beta} L^{(d+\alpha)(n+\ell)}}\end{equation}
as desired.
\end{proof}

We now deduce \cref{prop:maximum_upper} from \cref{lem:maximum_renormalization}. The proof will use the \emph{sharpness of the phase transition}, a fundamental result in percolation theory originally due to Menshikov \cite{MR852458} and Aizenman and Barsky \cite{aizenman1987sharpness} which is known to hold for arbitrary transitive weighted graphs \cite{aizenman1987sharpness,duminil2015new,1901.10363}. We state the theorem for the hierarchical lattice only as this is the only case relevant to us.

\begin{theorem}[Sharpness of the phase transition]
\label{thm:sharpness}
Let $d\geq 1$, $L\geq 2$, and let $J:\mathbb{H}^d_L \to [0,\infty)$ be a symmetric, integrable function. Then $\E_\beta |K| < \infty$ for every $0\leq \beta < \beta_c$.
%
\end{theorem}

\begin{proof}[Proof of \cref{prop:maximum_upper}]
Let $A$ be the constant from \cref{lem:maximum_renormalization}.
We have by sharpness of the phase transition that if $\beta<\beta_c$ then $\E_\beta |K| = \sup_n \E_\beta |K_n| < \infty$, and  it follows from \eqref{eq:susceptibility>max} that
\begin{equation}
\label{eq:sharpnessM}
\limsup_{n\to\infty} L^{-dn}M_n(\beta)^2 \leq \limsup_{n\to\infty} 4e \E_\beta |K_n| < \infty \qquad \text{ for every $0\leq \beta <\beta_c$.}
\end{equation}
On the other hand, if there were to exist $0\leq \beta<\beta_c$ and $n\geq 0$ such that $M_n(\beta)^2 \geq \frac{A}{\beta} L^{(d+\alpha)n}$ then we would have inductively by \cref{lem:maximum_renormalization} that
$M_{n+i\ell}(\beta)^2 \geq \frac{A}{\beta}L^{(d+\alpha)(n+i\ell)}$
for every $i\geq 1$, which would contradict \eqref{eq:sharpnessM}. Thus, we must instead have that
\[
M_n(\beta)^2 <  \frac{A}{\beta} L^{(d+\alpha)n}
\]
for every $n\geq 0$ and $0\leq \beta<\beta_c$. The claim follows by taking the limit as $\beta \uparrow \beta_c$.
\end{proof}

\begin{remark}
We conjecture that the inequality of \cref{prop:maximum_upper} is of the correct order if and only if $\alpha>d/3$. It is certainly not of the correct order when $\alpha < d/3$. Indeed, it is a consequence of the \emph{tree-graph inequality} method of Aizenman and Newman \cite{MR762034} (see in particular \cite[Equation 6.99]{grimmett2010percolation}) that
\begin{equation}
\P_\beta(|K_n| \geq m) \leq \frac{\sqrt{2}\E_\beta |K_n|}{m}\exp\left[-\frac{m}{4(\E_\beta |K_n|)^2}\right]
\end{equation}
for every $0\leq \beta <\infty$, $n\geq 1$, and $m \geq 1$, and hence by Markov's inequality that
\begin{equation}
\P_\beta(|K_n^\mathrm{max}| \geq m) \leq \frac{1}{m} \sum_{x\in \Lambda_n} \P_\beta(|K_n(x)| \geq m) \leq \frac{\sqrt{2} L^{dn} \E_\beta|K_n|}{m^{2}}\exp\left[-\frac{m}{4(\E_\beta |K_n|)^2}\right]
\end{equation}
for every $0\leq \beta <\infty$, $n\geq 1$, and $m \geq 1$. Applying \cref{thm:main} we deduce that there exist positive constants $A_1$ and $a$ such that
\begin{equation}
\P_{\beta_c}(|K_n^\mathrm{max}| \geq m) \leq \frac{A_1 L^{(d+\alpha)n}}{m^{2}}\exp\left[-a L^{-2\alpha n} m\right]
\end{equation}
for every  $n\geq 1$ and $m \geq 1$. It follows that there exists a constant $A_2$ such that 
\begin{equation}
M_n(\beta_c) \leq A_2 n L^{2\alpha n}
\end{equation}
for every $n\geq 1$, which is of lower order than the bound of \cref{prop:maximum_upper} when $\alpha<d/3$.
\end{remark}

\subsection{Upper bounds on the restricted two-point function}
\label{sec:restricted_upper}

We now apply the results of \cref{sec:maximum} to study the expected volume of the cluster of the origin within an ultrametric ball. 

\begin{prop}
\label{prop:susceptibility_upper}
Let $J:\mathbb{H}^d_L \to [0,\infty)$ be a symmetric, integrable function, let $0<\alpha<d$, and suppose that there exists a positive constant $c$ such that $J(x) \geq c\|x\|^{-d-\alpha}$ for every $x\in \mathbb{H}^d_L\setminus\{0\}$.
Then there exists a positive constant $A=A(d,L,\alpha,c)$ such that
\[\E_\beta |K_{n}| = \sum_{x\in \Lambda_n} \P_\beta\bigl(0\xleftrightarrow{\Lambda_n}x\bigr) \leq \frac{A}{\beta_c} L^{\alpha n} \]
for every $0<\beta \leq \beta_c$ and $n \geq 0$.
\end{prop}

We will deduce \cref{prop:susceptibility_upper} from the following renormalization lemma.

\begin{lemma}[Renormalization of the restricted susceptibility]
\label{lem:susceptibility_renormalization}
Let $J:\mathbb{H}^d_L \to [0,\infty)$ be a symmetric, integrable function, let $0<\alpha<d$, and suppose that there exists a positive constant $c$ such that $J(x) \geq c\|x\|^{-d-\alpha}$ for every $x\in \mathbb{H}^d_L\setminus\{0\}$. Then there exists a positive constant $a=a(d,L,\alpha,c)$ such that
\[\E_\beta |K_{n+1}| \geq \sum_{x\in \Lambda_{n+1} \setminus \Lambda_n} \P_\beta(0 \xleftrightarrow{\Lambda_{n+1}} x) \geq a \beta L^{-\alpha n} \left(\E_\beta |K_{n}|\right)^2\]
for every $0<\beta \leq \beta_c$ and $n \geq 0$.
\end{lemma}

In order to prove this lemma, we first use \cref{prop:maximum_upper,thm:universaltightness} to prove the following supporting technical estimate. We write $a \wedge b = \min\{a,b\}$.

\begin{lemma}[Truncating at the typical maximum]
\label{lem:truncated_susceptibility}
Let $J:\mathbb{H}^d_L \to [0,\infty)$ be a symmetric, integrable function, let $0<\alpha<d$, and suppose that there exists a positive constant $c$ such that $J(x) \geq c\|x\|^{-d-\alpha}$ for every $x\in \mathbb{H}^d_L\setminus\{0\}$.
Then there exists a positive constant $a=a(d,L,\alpha,c)$ such that
\[
\E_\beta \left[ |K_n| \wedge \bigl(\lambda L^{(d+\alpha)n/2}\bigr) \right] \geq a \lambda  \E_\beta |K_n|\]
for every $0<\beta \leq \beta_c$, $0<\lambda \leq 1$ and $n \geq 0$.
\end{lemma}

\begin{proof}[Proof of \cref{lem:truncated_susceptibility}]
Let $A$ be the constant from \cref{prop:maximum_upper}. Fix $n \geq 0$ and let $N=\lfloor 99 M_n \rfloor$.
Then we have that 
\[
\E_\beta |K_n| = \sum_{k = 1}^\infty \P_\beta(|K_n| \geq k) \qquad \text{ and } \qquad \E_\beta \left[|K_n| \wedge N\right] = \sum_{k = 1}^N \P_\beta(|K_n| \geq k),
\]
and hence by \cref{thm:universaltightness} that
\begin{align*}
\E_\beta |K_n| - \E_\beta \left[|K_n| \wedge N\right] = \sum_{k=N+1}^\infty \P_\beta(|K_n| \geq k)
&\leq e \P_\beta(|K_n| \geq M_n) \sum_{k= N+1}^\infty e^{-k/9M_n}.
\end{align*}
Since $N+1 \geq 99 M_n$, we can therefore compute by Markov's inequality that
\begin{align*}
\E_\beta |K_n| - \E_\beta \left[|K_n| \wedge N\right] &\leq \frac{e^{1-11}}{1-e^{-1/9M_n}} \P_\beta(|K_n| \geq M_n) \leq \frac{e^{-10}}{(1-e^{-1/9M_n}) M_n} \E_\beta|K_n| \leq \frac{1}{2} \E_\beta |K_n|,
\end{align*}
where the final inequality follows by calculus since $M_n \geq 2$. (Indeed, the optimal constant here is much smaller than $1/2$.) It follows that
\begin{equation}
\E_\beta \left[|K_n| \wedge N\right] \geq \frac{1}{2} \E_\beta |K_n|,
\end{equation}
and the claim follows from \cref{prop:maximum_upper} together with the trivial inequality $\E_\beta \left[|K_n| \wedge \bigl(\lambda N\bigr)\right] \geq \lambda \E_\beta \left[|K_n| \wedge  N \right]$, which holds 
for every $\lambda \in [0,1]$.
\end{proof}

We now apply \cref{lem:truncated_susceptibility} to prove \cref{lem:susceptibility_renormalization}.
For each $n\geq 0$ we write $\sC_n$ for the set of clusters of the percolation configuration inside $\Lambda_n$, so that $\sC_n$ is a random set of disjoint subsets of $\Lambda_n$ whose union is the entire ultrametric ball $\Lambda_n$.
Note that if $f$ is any function from subsets of $\Lambda_n$ to $\R$ that is isometry-invariant in the sense that $f(\gamma A) = f(A)$ for every $A \subseteq \Lambda_n$ and any isometry of $\Lambda_n$ then we have by symmetry  that
\begin{equation}
\label{eq:cluster_average}
\E_\beta \left[f(K_n)\right] = \frac{1}{|\Lambda_n|} \sum_{x\in \Lambda_n} \E_\beta \left[f(K_n(x))\right] = \frac{1}{|\Lambda_n|} \E_\beta \left[\sum_{C\in \sC_n} |C| f(C)\right],
\end{equation}
where the second equality follows by linearity of expectation.

\begin{proof}[Proof of \cref{lem:truncated_susceptibility}]
For each $n\geq 0$ write $h(n)=c \beta L^{-(d+\alpha)n}$, so that if $n\geq 1$ and $x,y$ satisfy $\|x-y\| = L^n$ then there is an edge connecting $x$ and $y$ with probability at least $1-e^{-h(n)}$. Fix $n\geq 0$. The ultrametric ball $\Lambda_{n+1}$ can be decomposed into $L^d$ copies of $\Lambda_n$. Call these copies $\Lambda^1_n,\ldots,\Lambda_n^{L^d}$, with the origin belonging to $\Lambda^1_n=\Lambda_n$, and for each $1\leq i \leq L^d$ let $\sC^i_n$ be the set of clusters of the restriction of the percolation configuration to $\Lambda^i_n$. (That is, $x,y\in \Lambda^i_n$ belong to the same element of $\sC^i_n$ if they are connected by a path inside $\Lambda^i_n$.) Let $\cF_n$ be the $\sigma$-algebra generated by all edges whose endpoints have distance at most $L^n$ in $\bbH^d_L$ and define
\[\tilde K_{n+1} = \Bigl\{x\in \Lambda_{n+1} \setminus \Lambda_n : 0 \xleftrightarrow{\Lambda_{n+1}} x\Bigr\}.\] Conditional on $\cF_n$, each cluster that belongs to $\sC^i_n$ for some $2\leq i \leq L^d$ is connected to $K_n$ by an edge with probability at least $1-e^{-h(n+1) |K_n| |C|}$, so that
\[
\E_\beta \left[|K_{n+1}| \mid \cF_n \right] \geq \sum_{i=2}^{L^d} \sum_{C \in \mathscr{C}^i_n} |C| \left(1-\exp\Bigl[-h(n+1) |K_n| |C|\Bigr]\right).
\]
Using the inequality $1-e^{-t} \geq (1-e^{-1}) (t \wedge 1)$ it follows that
\begin{align*}
\E_\beta \left[|\tilde K_{n+1}| \mid \cF_n \right] &\geq  (1-e^{-1})h(n+1) \sum_{i=2}^{L^d} \sum_{C \in \mathscr{C}^i_n} |C|\left(\frac{1}{h(n+1)} \wedge |K_n| |C|\right)\\
&\geq (1-e^{-1})h(n+1) \sum_{i=2}^{L^d} \sum_{C \in \mathscr{C}^i_n} |C| \left(\frac{1}{\sqrt{h(n+1)}} \wedge |C|\right)\left(\frac{1}{\sqrt{h(n+1)}} \wedge |K_n|\right).
\end{align*}
Taking expectations and using that $\sC^i_n$ and $K_n$ are independent for every $2 \leq i \leq L^d$ yields that
\begin{align*}
\E_\beta|\tilde K_{n+1}| &\geq (1-e^{-1})h(n+1) \E_\beta\left[\frac{1}{\sqrt{h(n+1)}} \wedge |K_n|\right] \sum_{i=2}^{L^d} \E_\beta\left[\sum_{C \in \sC^i_n} |C| \left(\frac{1}{\sqrt{h(n+1)}} \wedge |C|\right)\right]\\
&= (1-e^{-1})(L^d-1) h(n+1) L^{dn} \E_\beta\left[\frac{1}{\sqrt{h(n+1)}} \wedge |K_n|\right]^2,
\end{align*}
where the equality in the second line follows from \eqref{eq:cluster_average}. The claim now follows from \cref{lem:truncated_susceptibility} since $h(n+1)^{-1/2}$ is of order $L^{(d+\alpha)n/2}$.
\end{proof}

It remains to deduce \cref{prop:susceptibility_upper} from \cref{lem:maximum_renormalization}.

\begin{proof}[Proof of \cref{prop:susceptibility_upper}]
Let $a$ be the constant from \cref{lem:susceptibility_renormalization} and let $\beta<\beta_c$. If there were to exist $n\geq 0$ such that 
\[\E_\beta |K_n| \geq \frac{L^{\alpha (n+1)}}{a\beta} \]
then we would have by induction that
\[
\E_\beta |K_{m+1}| \geq a \beta L^{-\alpha m} \left(\frac{L^{\alpha(m+1)}}{a\beta}\right)^2 = \frac{L^{\alpha(m+2)}}{a\beta}
\]
for every $m \geq n$, contradicting the fact that $\E_\beta |K| = \sup_{m\geq 1} \E_\beta |K_m| < \infty$ by sharpness of the phase transition (\cref{thm:sharpness}). Thus, we must instead have that
\[\E_\beta |K_n| < \frac{L^{\alpha (n+1)}}{a\beta} \]
for every $0\leq \beta<\beta_c$ and $n\geq 0$, and the claim follows by continuity of $\E_\beta |K_n|$ as before.
\end{proof}

\subsection{Upper bounds on the unrestricted two-point function}

We now deduce upper bounds on the \emph{unrestricted} two-point function $\P_\beta(0\leftrightarrow x)$ from the corresponding upper bounds on the \emph{restricted} two-point function $\P_\beta(0\xleftrightarrow{\Lambda_n} x)$ proven in \cref{prop:susceptibility_upper}.

\begin{prop}
\label{prop:unrestricted_upper}
Let $J:\mathbb{H}^d_L \to [0,\infty)$ be a symmetric, integrable function, let $0<\alpha<d$, and suppose that there exist positive constants $c$ and $C$ such that $c \|x\|^{-d-\alpha} \leq J(x) \leq C\|x\|^{-d-\alpha}$ for every $x\in \mathbb{H}^d_L\setminus\{0\}$. Then there exists a positive constant $A=A(d,L,\alpha,c,C)$ such that
\[
\frac{1}{|\Lambda_n|} \sum_{x\in\Lambda_n}\P_{\beta_c}(0 \leftrightarrow{} x) \leq \frac{A}{\beta_c} L^{-(d-\alpha)n}
\]
for every $n \geq 0$.
\end{prop}

The proof of this proposition will employ the \emph{BK inequality} and notion of the \emph{disjoint occurrence} of events; we refer the reader to e.g.\ \cite[Chapter 2.3]{grimmett2010percolation} for relevant background.

\begin{proof}[Proof of \cref{prop:unrestricted_upper}]
We have by \cref{prop:susceptibility_upper} that there exists a positive constant $A_1=A_1(d,L,\alpha,c)$ such that
\begin{equation}
\label{eq:A1}
\sum_{x\in \Lambda_n} \P_{\beta_c}( 0 \xleftrightarrow{\Lambda_n} x) \leq \frac{A_1}{\beta_c} L^{\alpha n}.
\end{equation}
Fix $n \geq 0$ and  $x\in \bbH^d_L$ with $\langle x\rangle=L^n$. We have trivially that
\begin{equation}
\label{eq:mostdistantdecomp}
\P_{\beta_c}(0 \leftrightarrow x) \leq \P_{\beta_c}(0 \xleftrightarrow{\Lambda_{n}} x) + \sum_{k=0}^\infty \P_{\beta_c}(\{0 \xleftrightarrow{\Lambda_{n+k+1}} x\} \setminus  \{0 \xleftrightarrow{\Lambda_{n+k}} x\}).
\end{equation}
Let $k\geq 0$ and suppose that the event $\{0 \xleftrightarrow{\Lambda_{n+k+1}} x\} \setminus  \{0 \xleftrightarrow{\Lambda_{n+k}} x\}$ holds, so that there exists a simple path $\gamma$ connecting $0$ to $x$ that visits $\Lambda_{n+k+1}$ but not $\Lambda_{n+k+2}$. Let $z$ be the first point of $\Lambda_{n+k+1}$ visited by $\gamma$ and let $y$ be the point of $\Lambda_{n+k}$ visited immediately before $\gamma$ visits $z$. Then the portions of $\gamma$ up to first visiting $y$, the edge $\{y,z\}$, and the portion of $\gamma$ after first visiting $z$ are disjoint witnesses for the events $\{0 \xleftrightarrow{\Lambda_{n+k}} y\}$, $\{\{y,z\}$ open$\}$, and $\{z \xleftrightarrow{\Lambda_{n+k+1}} x\}$. Thus, we have by a union bound and the BK inequality that
\begin{align*}
&\P_{\beta_c}(\{0 \xleftrightarrow{\Lambda_{n+k+1}} x\} \setminus  \{0 \xleftrightarrow{\Lambda_{n+k}} x\})
\\
&\hspace{2cm}\leq \sum_{y \in \Lambda_{n+k}} \sum_{z \in \Lambda_{n+k+1}\setminus \Lambda_{n+k}} \P_{\beta_c}\left(\{0 \xleftrightarrow{\Lambda_{n+k}} y\}\circ\{\{y,z\} \text{ open}\}\circ\{z \xleftrightarrow{\Lambda_{n+k+1}} x\}\right)
\\
&\hspace{2cm}\leq \sum_{y \in \Lambda_{n+k}} \sum_{z \in \Lambda_{n+k+1}\setminus \Lambda_{n+k}} \P_{\beta_c}(0 \xleftrightarrow{\Lambda_{n+k}} y) \P_{\beta_c}(\{y,z\} \text{ open}) \P_{\beta_c}(z \xleftrightarrow{\Lambda_{n+k+1}} x) 
\end{align*}
for every $k\geq 0$. Using that $\{y,z\}$ is open with probability $1-e^{-\beta_c J(y-z)} \leq C \beta_c L^{-(d+\alpha)(n+k+1)}$, we deduce from \eqref{eq:A1} that
\begin{align*}
&\P_{\beta_c}(\{0 \xleftrightarrow{\Lambda_{n+k+1}} x\} \setminus  \{0 \xleftrightarrow{\Lambda_{n+k}} x\})\\
&\hspace{3cm}\leq C {\beta_c}  L^{-(d+\alpha)(n+k+1)}  \sum_{y \in \Lambda_{n+k}} \P_{\beta_c}(0 \xleftrightarrow{\Lambda_{n+k}} y) \sum_{z \in \Lambda_{n+k+1}}  \P_{\beta_c}(x \xleftrightarrow{\Lambda_{n+k+1}} z)\\
&\hspace{3cm}\leq \frac{C A_1^2}{{\beta_c}}L^{-(d+\alpha)(n+k+1)} L^{\alpha(n+k)} L^{\alpha(n+k+1)} \leq \frac{C A_1^2}{\beta_c} L^{-(d-\alpha)(n+k+1)}
\end{align*}
for every $k\geq 0$. Substituting this inequality into \eqref{eq:mostdistantdecomp}, summing over $k\geq 0$ and using \cref{cor:betabound} yields  that there exists a constant $A_2=A_2(d,L,\alpha,c,C)$ such that
\begin{equation}
\P_{\beta_c}(0\leftrightarrow x) \leq \P_{\beta_c}(0 \xleftrightarrow{\Lambda_{n}} x) + \frac{A_2}{\beta_c}\langle x \rangle^{-d+\alpha}
\end{equation}
for every $n\geq 0$ and every $x \in \bbH^d_L$ with $\langle x \rangle = L^{n}$. The claim follows by summing over $x\in \Lambda_n$ and applying \eqref{eq:A1} again.
\end{proof}

\begin{remark}
All the results of \cref{sec:maximum,sec:restricted_upper} apply equally well to long-range percolation on $\Z^d$ as defined in \cite{hutchcroft2020power}. In order to generalize the upper bound \cref{thm:main} to this setting, it would suffice to compare critical connection probabilities inside a box and inside the full space as we have done here. Unfortunately we are not aware of any good techniques to do this for long-range percolation on $\Z^d$ at present. 
\end{remark}

\subsection{Lower bounds}

We now prove the lower bounds of \cref{thm:main}. We begin with the following proposition.

\begin{prop}
\label{prop:lower_bound}
Let $J:\mathbb{H}^d_L \to [0,\infty)$ be a symmetric, integrable function, let $0<\alpha<d$, and suppose that there exists a constant $C$ such that $J(x) \leq C\|x\|^{-d-\alpha}$ for every $x\in \mathbb{H}^d_L\setminus\{0\}$. Then 
\[
\beta_c \geq \frac{L^\alpha-1}{C} \qquad \text{ and } \qquad \E_{\beta_c} |K_n| = \sum_{x\in \Lambda_n} \P_{\beta_c}(0 \xleftrightarrow{\Lambda_n} x) \geq L^{\alpha n} 
\]
for every $n\geq 0$.
\end{prop}

\begin{proof}[Proof of \cref{prop:lower_bound}] 
Following Duminil-Copin and Tassion \cite{duminil2015new}, we define for each finite set $S$ and $\beta \geq 0$ the quantity
\[
\phi_\beta(S) = \sum_{x\in S} \sum_{y \notin S} \Bigl(1-e^{-\beta J(y-x)}\Bigr)\P_\beta(0 \xleftrightarrow{S} x).
\]
It is proven in \cite{duminil2015new} that the critical parameter $\beta_c$ can be characterised alternatively as
\[\beta_c = \inf\Bigl\{ \beta \geq 0 : \phi_\beta(S) \geq 1 \text{ for every finite $S$}\Bigr\}.\]
Since $\phi_\beta(S)$ is a continuous function of $\beta$ when $S$ is finite, it follows in particular that $\phi_{\beta_c}(S) \geq 1$ for every finite set $S$.
 Using the inequality $1-e^{-x} \leq x$ we have for each $x\in \Lambda_n$ that
\begin{align*}\sum_{y \notin \Lambda_n} \Bigl(1-e^{-\beta_c J(y-x)}\Bigr) \leq  \sum_{r=n}^\infty \sum_{y\in \Lambda_{r+1}\setminus \Lambda_{r}}\Bigl(1-e^{-\beta_c J(y-x)}\Bigr)
\leq \sum_{r=n}^\infty \sum_{y\in \Lambda_{r+1}\setminus \Lambda_{r}}\beta_c J(y-x)
\end{align*}
for every $n\geq 0$ and $x\in \Lambda_n$ and hence that
\begin{align*}
\sum_{y \notin \Lambda_n} \Bigl(1-e^{-\beta_c J(y-x)}\Bigr) &\leq C\beta_c \sum_{r=n}^\infty L^{d(r+1)} L^{-(r+1)(d+\alpha)} = \frac{C\beta_c}{1-L^{-\alpha}} L^{-\alpha(n+1)}
\end{align*}
for every $n\geq 0$ and $x\in \Lambda_n$ by our assumptions on $J$. Thus, we deduce that
\begin{align}
\sum_{x\in \Lambda_n} \P_{\beta_c}(0 \xleftrightarrow{\Lambda_n} x) \geq \frac{1-L^{-\alpha}}{C{\beta_c}} L^{\alpha(n+1)} \phi_{\beta_c}(\Lambda_n) \geq \frac{1-L^{-\alpha}}{C{\beta_c}} L^{\alpha(n+1)} 
\label{eq:lower_betac}
\end{align}
for every $n\geq 0$.  Noting that the left hand side is $1$ when $n=0$ it follows in particular that
\begin{align}
\beta_c \geq \frac{L^\alpha-1}{C}
\qquad\text{and hence that}\qquad
\sum_{x\in \Lambda_n} \P_{\beta_c}(0 \xleftrightarrow{\Lambda_n} x) \geq L^{\alpha n} 
\label{eq:lower_betac}
\end{align}
for every $n\geq 0$ as claimed.
\end{proof}

We now have all the ingredients in place to complete the proof of our main theorem.

\begin{proof}[Proof of \cref{thm:main}]
The upper bound follows immediately from \cref{prop:unrestricted_upper} together with the lower bound on $\beta_c$ provided by \cref{prop:lower_bound}. The lower bound follows from \cref{prop:lower_bound} together with \cref{lem:susceptibility_renormalization} which yields that there exist positive constants $a_1$ and $a_2$ such that
\[
\sum_{x\in \Lambda_{n+1} \setminus \Lambda_n} \P_{\beta_c}(0 \xleftrightarrow{\Lambda_{n+1}} x) \geq a_1 \beta_c L^{-\alpha n} \left(\E_{\beta_c} |K_n|\right)^2 \geq a_2 L^{\alpha n}
\]
for every $n\geq 0$ as claimed.
\end{proof}

\subsection{Consequences for the tail of the volume}
\label{subsec:volume}

We now apply the results of the previous sections to prove \cref{cor:triangle,thm:delta_lower}, which establish respectively that the model has mean-field critical behaviour when $\alpha<d/3$ and does \emph{not} have mean-field critical behaviour when $\alpha > d/3$.

\begin{proof}[Proof of \cref{cor:triangle}]
We want to prove that the diagrammatic sum
$\nabla_{\beta_c}:=\sum_{x,y \in \bbH^d_L} \P_{\beta_c}(0\leftrightarrow x) \P_{\beta_c}(x \leftrightarrow y) \P_{\beta_c}(y \leftrightarrow 0)$
is finite.
For each $n\geq 0$ define
\[
\nabla_n := \sum_{x,y\in \Lambda_n} \P_{\beta_c}(0\leftrightarrow x) \P_{\beta_c}(x \leftrightarrow y) \P_{\beta_c}(y \leftrightarrow 0).
\]
Then we have by symmetry that $\nabla_0=1$ and
\begin{multline*}
\nabla_{n+1} = \nabla_n + \sum_{x,y \in \Lambda_{n+1} \setminus \Lambda_n}\P_{\beta_c}(0\leftrightarrow x) \P_{\beta_c}(x \leftrightarrow y) \P_{\beta_c}(y \leftrightarrow 0) \\+ 2\sum_{x \in \Lambda_{n}} \sum_{y\in \Lambda_{n+1} \setminus \Lambda_n}\P_{\beta_c}(0\leftrightarrow x) \P_{\beta_c}(x \leftrightarrow y) \P_{\beta_c}(y \leftrightarrow 0)
\end{multline*}
for every $n \geq 0$. We have by \cref{cor:mainsymmetric} that there exists a constant $A$ such that
\begin{align*}
\sum_{x,y \in \Lambda_{n+1} \setminus \Lambda_n}\P_{\beta_c}(0\leftrightarrow x) \P_{\beta_c}(x \leftrightarrow y) \P_{\beta_c}(y \leftrightarrow 0) &\leq A^2 L^{-2(d-\alpha)(n+1)} \sum_{x,y \in \Lambda_{n+1}}\P_{\beta_c}(x \leftrightarrow y) \\
&\leq A^3 L^{-2(d-\alpha)(n+1)} L^{d(n+1)} L^{\alpha(n+1)} = A^3 L^{(3\alpha -d)(n+1)}
\end{align*}
and
\begin{align*}\sum_{x \in \Lambda_{n}} \sum_{y\in \Lambda_{n+1} \setminus \Lambda_n}\P_{\beta_c}(0\leftrightarrow x) \P_{\beta_c}(x \leftrightarrow y) \P_{\beta_c}(y \leftrightarrow 0) &\leq A^2 L^{-2(d-\alpha)(n+1)} \sum_{x\in \Lambda_n} \P_{\beta_c}(0\leftrightarrow x)\\
&\leq A^3 L^{(3\alpha-d)(n+1)}
\end{align*}
for every $n\geq 0$. It follows by induction that
\begin{equation}
\label{eq:triangle_growth}
\nabla_{n+1} \leq \nabla_n + 3 A^3 L^{(3\alpha-d)(n+1)} \leq 3A^3 \sum_{k=0}^{n+1} L^{(3\alpha-d)(n+1)}
\end{equation}
for every $n\geq 0$, and hence that $\nabla_{\beta_c} = \lim_{n\to\infty} \nabla_n < \infty$ when $\alpha<d/3$ as claimed.
\end{proof}

\begin{remark}
In forthcoming work \cite{HutchcroftTriangle} we give a new, more quantitative derivation of mean-field critical behaviour from the triangle condition, which allows us to deduce from \eqref{eq:triangle_growth} that mean-field critical behaviour holds up to polylogarithmic factors when $\alpha=d/3$.  
\end{remark}

We now prove \cref{thm:delta_lower}. The proof will rely on the following special case of the rigorous hyperscaling inequality of \cite[Theorem 2.1]{hutchcroft2020power}, which is a consequence of \cref{thm:universaltightness}.

\begin{thm}
\label{thm:hyperscalingsimple}
There exists a universal constant $C$ such that the following holds. 
 Let $d \geq 1$, $L\geq 2$, and let $J:\bbH^d_L \to [0,\infty)$ be symmetric and integrable. Let $\beta \geq 0$, and suppose that there exist constants $A<\infty$ and $0 \leq \theta \leq 1/2$ such that $\P_\beta( |K| \geq n) \leq A n^{-\theta}$ for every $\lambda>0$. Then
\[
\sum_{x \in \Lambda} \P_\beta(0\leftrightarrow x) 
\leq 
C A^{2 /(1+\theta)} |\Lambda|^{(1-\theta)/(1+\theta)}
\]
for every every finite set $\Lambda \subseteq \bbH^d_L$.
\end{thm}

\begin{proof}[Proof of \cref{thm:delta_lower}]
Let $\alpha >d/3$ and suppose for contradiction that the exponent $\delta$ is well-defined and satisfies $\delta<(d+\alpha)/(d-\alpha)$. Thus, if we fix $\theta$ such that $(d-\alpha)/(d+\alpha)  <\theta < (1/\delta) \wedge (1/2)$ then there exists a positive constant $A_1$ such that
\[
\P_{\beta_c}(|K|\geq n) \leq A_1 n^{-\theta}
\]
for every $n\geq 1$. It follows \cref{thm:hyperscalingsimple} that there exists a constant $A_2$ such that
\begin{equation}
\label{eq:theta_contradiction}
\E_{\beta_c} |K_n| \leq A_2 |\Lambda_n|^{(1-\theta)/(1+\theta)} = A_2 L^{d(1-\theta)/(1+\theta) n}
\end{equation}
for every $n\geq 0$. Since $\theta> (d-\alpha)/(d+\alpha)$ we have that $(1-\theta)/(1+\theta) < \alpha/d$ so that \eqref{eq:theta_contradiction} contradicts \cref{thm:main} when $n$ is sufficiently large.
\end{proof}

\subsection*{Acknowledgments} This research was supported by ERC starting grant 804166 (SPRS). We thank Gordon Slade for helpful comments on a previous version of the manuscript.

 \setstretch{1}
 \footnotesize{
  \bibliographystyle{abbrv}
  \bibliography{unimodularthesis.bib}
  }

\end{document}